\documentclass[a4paper]{amsart}
\usepackage{amsmath,amssymb,amsthm,amscd,mathrsfs}
\usepackage[all]{xy}
\usepackage{hyperref}
\usepackage{leftidx}
\usepackage[usenames,dvipsnames,svgnames,table]{xcolor}
\usepackage{amsrefs}
\usepackage{marginnote}
\usepackage[british]{babel}

\newcommand{\C}{\mathcal{C}}

\DeclareMathOperator{\Hom}{Hom}
\DeclareMathOperator{\RHom}{\mathbf{R}Hom}

\DeclareMathOperator{\Ext}{Ext}
\DeclareMathOperator{\Tor}{Tor}

\DeclareMathOperator{\Cone}{Cone}

\newcommand{\perf}{\mathrm{perf}}

\newcommand{\op}{\mathrm{o}}
\newcommand{\pd}{\mathrm{pd}}
\newcommand{\fd}{\mathrm{fd}}
\newcommand{\id}{\mathrm{inj.dim}}

\newcommand{\tens}[1]{%
  \mathbin{\mathop{\otimes}\displaylimits_{#1}}%
}

\newcommand{\ModA}{{\mathrm{Mod}(A)}}
\newcommand{\modA}{{\mathrm{mod}(A)}}

\newcommand{\modB}{{\mathrm{mod}(B)}}

\newcommand{\proj}[1]{\hbox{\rm proj}{#1}}

\newcommand{\stmod}[1]{\underline{\mathrm{mod}}\,{#1}}
\newcommand{\modd}[1]{\mathrm{mod}{#1}}

\newcommand{\rad}{\mathrm{rad}}

\newcommand{\Cpx}{\mathbf{C}}

\newcommand{\Htp}{\mathbf{K}}

\newcommand{\Der}{\mathbf{D}}

\newcommand{\Dsing}{\mathbf{D}_\mathrm{sg}}

\newtheorem{theorem}{Theorem}[section]
\newtheorem*{introthm}{Theorem}
\newtheorem*{introcor}{Corollary}
\newtheorem{lemma}[theorem]{Lemma}

\newtheorem{proposition}[theorem]{Proposition}
\newtheorem{corollary}[theorem]{Corollary}

\theoremstyle{definition}
\newtheorem{definition}[theorem]{Definition}
\newtheorem{fact}[theorem]{Fact}
\newtheorem{remark}[theorem]{Remark}

\newtheorem{example}[theorem]{Example}

\newtheorem{ipg}[theorem]{}

\begin{document}

\title{On singular equivalences of Morita type with level and Gorenstein algebras}


\author{Georgios Dalezios}
\address{Department of Mathematics, University of Athens, Athens 15784, Greece}
\email{gdalezios@math.uoa.gr}

\thanks{This research is co-financed by Greece and the European Union (European Social Fund- ESF) through the
Operational Programme «Human Resources Development, Education and Lifelong Learning» in the context
of the project “Reinforcement of Postdoctoral Researchers - 2nd Cycle” (MIS-5033021), implemented by the
State Scholarships Foundation (IKY)}

\begin{abstract}
Rickard proved that for certain self-injective algebras, a stable equivalence induced from an exact functor is a stable equivalence of Morita type, in the sense of Brou\'{e}. In this paper we study singular equivalences of finite dimensional algebras induced from tensor product functors. We prove that for certain Gorenstein algebras, a singular equivalence induced from tensoring with a suitable complex of bimodules, induces a singular equivalence of Morita type with level, in the sense of Wang. This recovers Rickard's theorem in the self-injective case.
\end{abstract}

\maketitle

\section{Introduction}
\label{sec:intro}
If $A$ is a finite dimensional algebra over a field, the study of its stable module category $\underline{\mathrm{mod}}(A)$, which is the additive quotient of the finitely generated $A$--modules modulo the projectives, has its origins in the (non-semisimple) representation theory of finite groups. In case $A$ is self-injective then $\underline{\mathrm{mod}}(A)$ is a triangulated category, see Happel~\cite[I.2]{Happel}, therefore techniques from the realm of triangulated categories can be used to study representations of finite groups and more generally self-injective algebras.

For a general left noetherian ring $A$ the category $\underline{\mathrm{mod}}(A)$ is not necessarily triangulated, but its \textit{singularity category} $\Dsing(A):=\Der^{b}(\mathrm{mod}(A))/\Htp^{b}(\proj(A))$ is. Note that this construction is analogous to that of $\underline{\mathrm{mod}}(A)$, namely we take the Verdier quotient of $\Der^{b}(\mathrm{mod}(A))$ modulo those complexes that have finite projective dimension, in a standard sense. Results of Buchweitz~\cite{Buch} and Rickard~\cite{JRc89} tell us that in case $A$ is self-injective, the canonical map $\underline{\mathrm{mod}}(A)\rightarrow \Dsing(A)$ is a triangulated equivalence. In fact, Buchweitz in \cite{Buch} proved more generally that in case $A$ is Gorenstein (i.e., two sided noetherian with finite injective dimension over itself on both sides), the canonical map $\underline{\mathrm{MCM}}(A)\rightarrow \Dsing(A)$ is a triangulated equivalence. Here $\underline{\mathrm{MCM}}(A)$ denotes the stable category of maximal Cohen-Macaulay $A$--modules.

In the spirit of Morita theory, it is an honest question to ask when two rings have equivalent stable module categories. However, an arbitrary equivalence of this kind does not preserve important properties of the rings in question. For example, if $k$ is a field, $A=k[x]/(x^{2})$ and $B$ is a triangular matrix algebra with entries in $k$, then $\underline{\mathrm{mod}}(A)\cong \mathrm{mod}(k)\cong\underline{\mathrm{mod}}(B)$, but $A$ is self-injective with infinite global dimension while $B$ does not satisfy any of these properties. An appropriate notion of equivalence between stable module categories is that of a ``stable equivalence of Morita type'', introduced by Brou\'{e} \cite[5.A~Definition]{Broue}. 

We recall Brou\'{e}'s definition: Given a field $k$ and two finite dimensional $k$--algebras $A$ and $B$, we say that a pair of bimodules $(\leftidx{_{B}}M_{A},\leftidx{_{A}}N_{B})$ defines a \textit{stable equivalence of Morita type} between $A$ and $B$, if $M$ (resp., $N$) is finitely generated and projective over $B$ and $A^{\op}$ (resp., over $A$ and $B^{\op}$), and if the following hold:
\begin{equation}
\label{Broue_dfn}
N\otimes_{B}M\cong A\,\,\, \mbox{in}\,\,\, \underline{\mathrm{mod}}(A^e)\,\,\,\,\,\, \mbox{and}\,\,\,\,\,\, M\otimes_{A}N\cong B\,\,\, \mbox{in}\,\,\, \underline{\mathrm{mod}}(B^e).
\end{equation}
Here $A^{e}$, resp., $B^{e}$, denotes the enveloping algebra of $A$, resp., $B$. In this situation there is an equivalence $M\otimes_{A}-\colon \underline{\mathrm{mod}}(A)\rightarrow\underline{\mathrm{mod}}(B)$ with inverse $N\otimes_{B}-$. These equivalences usually preserve important properties of the rings in question under mild assumptions, see for instance Liu and Xi \cite{Liu1,Liu2}.

We mention an interesting Theorem of Rickard~\cite{Rick} that we will generalize. It states that for self-injective $k$--algebras $A$ and $B$, whose semisimple quotients are separable, any stable equivalence induced from an exact functor $\leftidx{_{B}}M_{A}\otimes_{A}-\colon \modA\rightarrow \modB$, is necessarily a stable equivalence of Morita type. Thus the definition of stable equivalence of Morita type can be simplified for such self-injective algebras. The proof of this result makes use of the triangulated structure of the stable module categories of $A$ and $B$.

Brou\'{e}'s definition has been generalized by Chen and Sun~\cite{Chen-Sun}, and further by Wang~\cite{Wang}. Wang's definition only differs than that of Brou\'{e} in that the conditions in (\ref{Broue_dfn}) are now replaced by:
\begin{equation}
N\otimes_{B}M\cong \Omega_{A^{e}}^{l}(A)\,\,\, \mbox{in}\,\,\, \underline{\mathrm{mod}}(A^e)\,\,\,\,\,\, \mbox{and}\,\,\,\,\,\, M\otimes_{A}N\cong \Omega_{B^{e}}^{l}(B)\,\,\, \mbox{in}\,\,\, \underline{\mathrm{mod}}(B^e),\nonumber
\end{equation}
where $l\in\mathbb{N}$ and $\Omega_{A^{e}}(-)$, resp., $\Omega_{B^{e}}(-)$, denotes the syzygy endofunctor of the stable module category of $A^{e}$, resp., $B^{e}$. 
In this situation, there is a triangulated equivalence $M\otimes_{A}-\colon \Dsing(A)\rightarrow\Dsing(B)$ with inverse $\Sigma^{l}\circ(N\otimes_{B}-)$. Wang calls this a \textit{singular equivalence of Morita type with level $l$} between $A$ and $B$. This concept is relatively new but has attracted some attention, see for instance the articles \cite{Psa2,Ska,Wang2}.

In this paper, given finite dimensional algebras $A$ and $B$, we look at tensor product functors $F:=X\otimes_{A}^\textbf{L}-\colon \Der^{b}(\modA)\rightarrow \Der^{b}(\modB)$, where $X$ is a complex of $B$-$A^{\op}$--bimodules which is perfect over $B$ and over $A^{\op}$, and we are interested in necessary and sufficient conditions imposed on $X$ in order for the functor $F$ to induce a singular equivalence between $A$ and $B$. This approach is simple but gives some interesting results. For instance, Theorem~\ref{thm:suff_cond} is a ``bimodule version'' of a result of Oppermann-Psaroudakis-Stai~\cite[Prop.~3.7.1]{Psa-2018}, which recovers some known results from the literature on singular equivalences (see \ref{example:chen},~\ref{prop:idemp_without_level}) and gives some examples of singular equivalences of Morita type with level (see \ref{rmk_idempot_with_level},~\ref{example:morita_rings_level}).

Next, we look at singular equivalences of Morita type with level for Gorenstein algebras. We obtain the following which is our main result:

\begin{introthm}
Let $k$ be a field and let $A$ and $B$ be finite dimensional Gorenstein $k$--algebras with separable semisimple quotients. Consider a complex $X$ of finitely generated $B$-$A^{\op}$--bimodules which is perfect over $B$ and $A^{\op}$. If the (well defined) functor \[X\otimes_{A}^\textbf{L}-\colon \Der^{b}(\modA)\rightarrow \Der^{b}(\modB)\] restricts to a singular equivalence, then it induces a singular equivalence of Morita type with level.
\end{introthm}

We note that separability is a minor technical assumption which is satisfied in most cases of interest, for example when $k$ is a perfect field; see \ref{subsec:Goresntein_rings}.

In terms of maximal Cohen-Macaulay modules over Gorenstein rings, we obtain the following. 

\begin{introcor}
Let $k$ be a field and let $A$ and $B$ be finite dimensional Gorenstein $k$--algebras with separable semisimple quotients. Consider a finitely generated $B$-$A^{\op}$--bimodule $M$ which is projective over $B$ and over $A^{\op}$. Denote $M^{\vee}:=\Hom_{B}(M,B)$.

If the (well-defined) functor $F:=M\otimes_{A}-\colon\mathrm{MCM}(A) \rightarrow\mathrm{MCM}(B)$ induces a triangulated equivalence $\overline{F}\colon\underline{\mathrm{MCM}}(A)\rightarrow\underline{\mathrm{MCM}}(B)$, then the pair $(M,\Omega^{l}_{A\otimes_{k}B^{\op}}M^{\vee})$ defines a singular equivalence of Morita type with level $l:=2\max\{\mathrm{vdim}(A),\mathrm{vdim}(B)\}$.
\end{introcor}
Here $\mathrm{vdim}$ denotes the so-called virtual dimension of a Gorenstein ring which equals its injective dimension as a module over itself (on either side), see \ref{subsec:Goresntein_rings}.

For a self-injective algebra all finitely generated modules are maximal Cohen-Macaulay and its virtual dimension is zero (\ref{subsec:Goresntein_rings}). Also, a singular equivalence of Morita type with level zero is just a stable equivalence of Morita type in the sense of Brou\'{e} (\ref{dfn-morita-type}). Thus we recover the following:

\textit{Rickard's Theorem. Let $k$ be a field and let $A$ and $B$ be finite dimensional self-injective $k$--algebras with separable semisimple quotients. Assume that there exists an exact additive functor $F:=M\otimes_{A}-\colon\mathrm{mod}(A) \rightarrow\mathrm{mod}(B)$ that induces a triangulated equivalence $\overline{F}\colon\underline{\mathrm{mod}}(A) \rightarrow\underline{\mathrm{mod}}(B)$. Then the pair $(M,\Hom_{B}(M,B))$ defines a stable equivalence of Morita type between $A$ and $B$.
}

This result is originally from \cite[Thm.~3.2]{Rick}. The reader may also consult the book of Zimmermann~\cite[Prop.~5.3.17]{Zim} or Dugas and Mart\'{\i}nez-Villa \cite{Dugas-ViIla}.

The Theorem and its Corollary are proved in Section~\ref{sec:sing_eq_of_Morita_type}, see \ref{proof_main} and \ref{proof_corollary}.

\section{Preliminaries}
\label{sec:preliminaries}

\subsection{Complexes}
\label{subsec:comp}
Let $A$ be a ring. Throughout the text, $\ModA$  stands for the category of left $A$--modules while right $A$--modules are understood as modules over the ring $A^{\op}$. We denote by $\modA$ the subcategory of $\ModA$ which consists of finitely presented $A$--modules, while $\proj(A)$ denotes the subcategory of $\modA$ which consists of projective $A$--modules.

$\Cpx(A)$ denotes the category of chain complexes of $A$--modules, with homological indexing. 
A complex $X$ is called bounded above (resp., below)  if $X_{>>0}=0$ (resp., $X_{<<0}=0$), and is called homologically bounded above (resp., below) if $H_{>>0}(X)=0$ (resp., $H_{<<0}(X)=0$). A complex which is (homologically) bounded above and below is just called (homologically) bounded.

We recall a few things on homological dimensions of complexes. We say that a homologically bounded below complex $X$ in $\C(A)$ has $\pd_{A}X\leq n$ (resp., $\fd_{A}X\leq n$), for some $n\in\mathbb{N}$, if there exists a complex $P$ of projective (resp., flat) $A$--modules and a quasi-isomorphism $P\xrightarrow{\sim} X$, where $P_{j}=0$ for all $j>n$. Similarly, we say that a homologically bounded above complex $X$ in $\C(A)$ has $\id_{A}X\leq n$, for some $n\in\mathbb{N}$, if there exists a complex of injectives and a quasi-isomorphism $X\xrightarrow{\sim} I$, where $I_{j}=0$ for all $j<-n$.

We denote by $\Htp(A)$ the homotopy category of complexes of $A$--modules and by $\Der(A)$ its derived category. We denote by $\Sigma(-)$ the shift endofunctor and by $\Sigma^{n}(-)$ the n-fold composition of $\Sigma$ with itself.

\subsection{Resolutions with bimodules}
\label{subsec:res_bi} The following facts are well-known. The reader may consult for instance \cite[Ch.~7]{dcm}. 

Let $k$ be a commutative ring and let $A$ and $B$ be $k$--algebras. 

If $A$ is projective over $k$, then $B\otimes_{k}A^{\op}$ is projective over $B$, hence (semi)projective resolutions\footnote{The word ``semi'' refers to resolutions of unbounded complexes and is used here to simplify notation in what follows. The reader can restrict to classical projective resolutions but the necessary boundedness assumptions should be made.} over $B\otimes_{k}A^{\op}$ restrict to (semi)projective resolutions over $B$. In this case, the derived functor,
\[\RHom_{B}(-,-)\colon\Der(B\otimes_{k}A^{\op})^{\mathrm{op}}\times\Der(B)\rightarrow\Der(A),\]
may be computed by $\RHom_{B}(-,-)\cong\Hom_{B}(\mathscr{P}(-),-)$, where $\mathscr{P}(-)$ is the (semi)projective resolution endofunctor of $\Htp(B\otimes_{k}A^{\op})$.

If $B$ is flat over $k$, then $B\otimes_{k}A^{\op}$ is flat over $A^{\op}$, hence (semi)flat resolutions over $B\otimes_{k}A^{\op}$ restrict to (semi)flat resolutions over $A^{\op}$. In this case, the derived functor,
\[-\otimes_{k}^{\textbf{L}}-\colon \Der(B\otimes_{k}A^{\op})\times\Der(A)\rightarrow\Der(B),\]
may be computed by $-\otimes_{A}^{\textbf{L}}-\cong\mathscr{F}(-)\otimes_{A}-$, where $\mathscr{F}(-)$ is the (semi)flat resolution endofunctor on $\Htp(B\otimes_{k}A^{\op})$.

\subsection{Singularity categories}
\label{subsec:sing_cats}
In this subsection $A$ denotes a finite dimensional algebra over a field $k$. 
Denote by $\Htp^{b}(\modA)$ the bounded homotopy category of complexes which are degreewise finitely generated and by $\Der^{\mathrm{b}}(\mathrm{mod}(A))$ its derived category. Note that the objects of $\Der^{\mathrm{b}}(\mathrm{mod}(A))$ are chain complexes in $\mathrm{mod}(A)$ which are homologically bounded.

\begin{fact}
\label{fact_perf}
The following are equivalent for a homologically bounded below and degreewise finitely generated complex $X$ in $\Cpx(A)$.
\begin{itemize}
\item[(i)] $X$ is isomorphic in $\Der(A)$ to a bounded complex of finitely generated projective $A$--modules.
\item[(ii)] $\pd_{A}X\leq n$, for some $n\in\mathbb{Z}$.
\item[(iii)] For any homologically bounded complex $Y$, the complex $\RHom_{A}(X,Y)$ is homologically bounded.
\item[(iv)] $X$ belongs to the bounded homotopy category $\Htp^{\mathrm{b}}(\proj A)$.
\end{itemize}
\end{fact}

\begin{proof}
See for instance \cite[Sec.~2P]{avf}, or \cite[Th.~8.1.14]{dcm}.
\end{proof}

\begin{definition}
\label{dfn_perf}
A complex in $\Cpx(A)$ which satisfies any of the equivalent conditions of Fact \ref{fact_perf} is called \textit{perfect}. We denote the category of perfect complexes by $\perf(A)$.
\end{definition}

\begin{definition}
The \textit{singularity category} of $A$ is the Verdier quotient $\Dsing(A):=\Der^{b}(\mathrm{mod}(A))/\perf(A)$.
\end{definition}

We also recall the \textit{stable module category} of $A$. It is the additive quotient $\stmod(A):=\modA/\sim$, where its objects are the same as those of $\modd(A)$ and two parallel morphisms are identified if they factor through a projective module. The syzygy endofunctor $\Omega_{A}(-)$ of $\stmod(A)$ maps an $A$--module $M$ to the kernel of a projective presentation of $M$. We denote by $\Omega^{n}_{A}(-)$ the $n$-fold composition of $\Omega_{A}(-)$ with itself.

There is a natural map $\modd(A)\rightarrow \Dsing(A)$ (that takes a module to its stalk complex, concentrated in degree zero), which factors through the stable module category to give a map $\pi:\stmod(A)\rightarrow \Dsing(A)$. It is an important property of the singularity category that the following diagram is commutative,

\begin{displaymath}
 \xymatrix@C=3pc{
\stmod(A)\ar[d]_-{\pi} \ar[r]^-{\Omega_{A}^{n}(-)}& \stmod(A) \ar[d]^-{\pi} \\
 \Dsing(A) \ar[r]^-{\Sigma^{-n}} & \Dsing(A).
 }
\end{displaymath}
See for instance \cite[page~925]{chenrad}.

\subsection{Special types of algebras}
\label{subsec:Goresntein_rings}
If $k$ is a commutative ring and $A$ is a $k$--algebra then $A^{e}:=A\otimes_{k}A^{\op}$ denotes the enveloping algebra of $A$.

We recall that a semisimple $k$--algebra ($k$ a field) is called \textit{separable} if its extension of scalars over any field extension of $k$ remains semisimple (see e.g. \cite[Ch.~1.7]{Reiner}). Now let $A$ be a finite dimensional $k$--algebra and assume that $A/\rad(A)$ is separable. In this case $A/\rad(A)\otimes_{k}A^{\op}/\rad(A^{\op})$ is semisimple and there is an isomorphism $A/\rad(A)\otimes_{k}A/\rad(A)\cong(A\otimes_{k} A^{\op})/\rad(A\otimes_{k}A^{\op})$ over $A^{e}$. This is practical because it implies that simple $A^{e}$--modules are summands of modules of the form $S\otimes_{k}S'$ where $S'$ is simple over $A^{\op}$ and $S$ is simple over $A$. For details see \cite[Cor.~5.3.10]{Zim}.

We recall that a finite dimensional $k$--algebra $A$ is said to have \textit{infinite global dimension} if there exists an $A$--module of infinite projective dimension. In \cite[1.5]{Ha} it is proved that if $\pd_{A^{e}}A<\infty$ then $A$ has finite global dimension.

We recall that a ring $A$ is called \textit{Gorenstein} if it is noetherian on both sides and has finite injective dimension as a module over itself on both sides. In this case from \cite{Zaks} we know that $\id_{A}A=n=\id_{A^{\op}}A$, for some natural number $n:=\mathrm{vdim}(A)$ which is called the \textit{virtual dimension} of $A$. In case $n=0$ the ring $A$ is called \textit{self-injective}.
For a Gorenstein ring $A$, we consider
\[\mathrm{MCM}(A):=\{M\in\modd(A)\, |\, \Ext^{\geqslant 1}_{A}(M,A)=0\},\]
the class of \textit{maximal Cohen-Macaulay} $A$--modules, which coincides with the class of the so-called (finitely generated) \textit{Gorenstein-projective} $A$--modules.  It is well known that this category is additive Frobenius, thus its stable category $\underline{\mathrm{MCM}}(A)$, which is defined in analogy with the stable module category mentioned above, is triangulated (see Happel~\cite[I.2]{Happel}). In case $A$ is self-injective we have $\mathrm{MCM}(A)=\modd(A)$. 

For a survey of Gorenstein homological algebra in the context of artinian algebras the interested reader may consult X.-W.~Chen \cite{Chen-Survey}.

\section{Singular equivalences induced from tensor products}
\label{sec:sing_eq_general}
We will make use of the following:
\begin{ipg}
\label{Setup}
\textbf{Setup} Let $k$ be a field and let $A$ and $B$ be finite dimensional $k$--algebras.
\end{ipg}

\begin{proposition}
\label{prop:adjoints}
Under Setup~\ref{Setup}, let $X$ be a complex of finitely generated $B$-$A^{\op}$--bimodules which is perfect over $B$ and over $A^{\op}.$ Then the following hold:
\begin{itemize}
\item[(i)] There exists an adjoint pair of exact (trianglulated) functors:
\begin{displaymath}
\label{adjunction1}
 \xymatrix@C=6pc{
 \mathbf{D}^b(\modA) \ar@<+0.7ex>[r]^-{F:=_{B}X_{A}\otimes^{\textbf{L}}_{A}-}   &  \mathbf{D}^b(\modB)  \ar@<+0.7ex>[l]^-{G:=\RHom_{B}(X,-)},
 }
 \tag{$\clubsuit$}
\end{displaymath}
where $F$ and $G$ may be computed by considering a projective resolution of $X$ over $B\otimes_{k}A^{\op}$.
\item[(ii)] There exists an isomorphism of functors $G\cong\RHom_{B}(X,B)\otimes_{B}^{\textbf{L}}-$.
\item[(iii)] The complex of $A$-$B^{\op}$--bimodules $X^{\vee}:=\RHom_{B}(X,B)$ is perfect over $B^{\op}$ (but not necessarily perfect over $A$).
\end{itemize}
\end{proposition}

\begin{proof} We omit the proof since it is standard. We just note that the assumptions on $X$ imply, in particular, that $F$ and $G$ are well-defined.
\end{proof}

\begin{remark}
\label{rem:adjoints}
(i) The functor $F$ in Proposition~\ref{prop:adjoints} maps $\perf(A)$ to $\perf(B)$. Indeed, $F(A)=X$ is perfect over $B$ and since $\mathrm{perf}(A)$ is generated by $A$ as a thick subcategory and $F$ is an exact functor, the result follows.

(ii) In general it is not true that the functor $G$ in Proposition~\ref{prop:adjoints} maps $\perf(B)$ to $\perf(A)$, unless we know that $G(B):=X^{\vee}$ is in $\perf(A)$. Indeed, in this case, the result follows since $\mathrm{perf}(B)$ is generated by $B$ as a thick subcategory and $G$ is an exact functor.
\end{remark}

A key technical point is that the functor $G$ in Proposition~\ref{prop:adjoints} maps $\perf(B)$ to $\perf(A)$ under the assumption that $A$ and $B$ are Gorenstein algebras.

\begin{lemma}
\label{lemma:Gor_adj}
Under Setup \ref{Setup}, assume in addition that $A$ and $B$ are Gorenstein $k$--algebras. Then the functor $G$ in Proposition~\ref{prop:adjoints} maps $\perf(B)$ to $\perf(A)$.
\end{lemma}

\begin{proof}
From Remark \ref{rem:adjoints}~(ii), it suffices to prove that $G(B)=\RHom_{B}(X,B)=:X^{\vee}$ belongs to $\perf(A)$. First note that if $\mathscr{P}\rightarrow X$ is a projective resolution of $X$ over $B\otimes_{k}A^{\op}$, then $X^{\vee}\cong\Hom_{B}(\mathscr{P},B)$ is degreewise finitely generated over $A$, thus we only need to prove that $\pd_{A}X^{\vee}<\infty$. To see this consider the natural isomorphism:
\[\RHom_{A}(-,\RHom_{B}(X,B))\cong\RHom_{B}(X\otimes^{\mathbf{L}}_{A}-,B).\]
If we input a homologically bounded complex of $A$--modules in this isomorprhism, on the right hand side we will obtain a homologically bounded complex (since $\pd_{A^{\op}}X<\infty$ and $\id_{B}B<\infty$). This shows that $\id_{A} X^{\vee}<\infty$. Since $A$ is Gorenstein this is equivalent to $\pd_{A}X^{\vee}<\infty$, which finishes the proof.
\end{proof}

In Theorem \ref{thm:suff_cond} below, we give necessary and sufficient conditions on the functor $F$ in Proposition~\ref{prop:adjoints} to induce a singular equivalence. We will need the following Lemma, which in the module case is a known result of Auslander and Reiten~\cite[Prop.~5.3.11]{Zim}. It is for this reason that we need to restrict to finite dimensional algebras with separable semisimple quotient in the sequel. 

\begin{lemma}
\label{lemma:perf_env}
Let $A$ be a finite dimensional $k$--algebra where $k$ is a field. Let $C$ be a complex of finitely generated $A^{e}$--modules. Consider the following:
\begin{itemize}
\item[(i)] $C\in\perf(A^{e})$.
\item[(ii)] For all complexes $Z\in\Der^b(\modA)$; $C\otimes_{A}^{\textbf{L}}Z\in\perf(A)$.
\item[(iii)] For all simple $A$--modules $N$; $C\otimes_{A}^{\mathbf{L}}N\in\perf(A)$.
\end{itemize}
Then the implications $(i)\Rightarrow(ii)\Rightarrow(iii)$ hold. Moreover, in case $A/\rad(A)$ is separable all statements are equivalent.
\end{lemma}

\begin{proof}
$(i)\Rightarrow (ii)$. We note that for any complex $Z$ in $\Der^b(\modA)$ the functor $-\otimes_{A}^{\textbf{L}}Z$ maps $A^e$ to $A^e\otimes_{A}^{\textbf{L}}Z\cong A\otimes_{k}Z\in\mathrm{perf}(A)$. Since $A^e$ generates $\perf(A^e)$ as a thick subcategory and the functor $-\otimes_{A}^{\textbf{L}}Z$ is exact the result follows.

We now prove that $(iii)\Rightarrow (i)$ under the assumption that $A/\rad(A)$ is separable. Since $A^{e}$ is noetherian it suffices to show that $\fd_{A^e}C<\infty$. For this, it suffices to show that for any simple $A^e$--module $S$, the complex $S\otimes_{A^e}^{\mathbf{L}}C$ is homologically bounded. From the assumption that $A/\rad(A)$ is separable, as we recalled in \ref{subsec:Goresntein_rings}, we know that all simple $A^e$--modules are direct summands of modules of the form $S'\otimes_{k}S''$, where $S'$ is a simple $A$--module and $S''$ is a simple $A^{\op}$--module, hence it suffices to prove that for such modules the complex $(S'\otimes_{k}S'')\otimes_{A^e}^{\mathbf{L}}C$ is homologically bounded.
We have a $k$-linear isomorphism of complexes, \[(S'\otimes_{k}S'')\otimes_{A^{e}}^{\mathbf{L}} C\cong S''\otimes_{A}^{\mathbf{L}} C\otimes_{A}^{\mathbf{L}} S'.\] By assumption, the complex $C\otimes_{A}^{\mathbf{L}} S'$ is in $\perf(A)$, hence the complex  $S''\otimes_{A}^{\mathbf{L}} (C\otimes_{A}^{\mathbf{L}} S')$ is homologically bounded, which finishes the proof.
\end{proof}

\begin{theorem}
\label{thm:suff_cond}
Let $k$ be a field, let $A$ and $B$ be finite dimensional $k$--algebras, and let $X$ be a complex of finitely generated $B$-$A^{\op}$--bimodules which is perfect over $B$ and over $A^{\op}$. Assume that $X^{\vee}:=\RHom_{B}(M,B)$ is a perfect complex of $A$--modules.

If the natural maps $A\rightarrow\RHom_{B}(X,X)$ in $\mathbf{D}_{\mathrm{sg}}(A^e)$ and $\RHom_{A}(X^{\vee},X^{\vee})\rightarrow B$ in $\mathbf{D}_{\mathrm{sg}}(B^e)$ are isomorphisms, then the functor $X\otimes^{\textbf{L}}_{A}-\colon\mathbf{D}_{\mathrm{sg}}(A)\rightarrow \mathbf{D}_{\mathrm{sg}}(B)$ is a triangulated equivalence (with inverse $X^{\vee}\otimes^{\textbf{L}}_{B}-$). The converse holds under the assumption that $A/\rad(A)$ and $B/\rad(B)$ are separable.
\end{theorem}

\begin{proof}
From Remark~\ref{rem:adjoints} and our assumptions we have that the adjunction (\ref{adjunction1}) from Proposition~\ref{prop:adjoints} restricts to 
\begin{displaymath}
 \xymatrix@C=5pc{
 \Dsing(A) \ar@<+0.7ex>[r]^-{\bar{F}:=\leftidx{_{B}}X_{A}\otimes^{\textbf{L}}_{A}-}   & \Dsing(B).  \ar@<+0.7ex>[l]^-{\bar{G}:=\leftidx{_{A}}X^{\vee}_{B}\otimes^{\textbf{L}}_{B}-}
  }
\end{displaymath}
We investigate when the unit $\bar{\eta}$ and the counit $\bar{\epsilon}$ of $(\bar{F},\bar{G})$ are isomorphisms. 
Consider a complex $Y$ in $\mathbf{D}^b(\modB)$. If $\rho\colon\mathscr{P}\xrightarrow{\sim}X$ is a projective resolution of $X$ over $B\otimes_{k}A^{\op}$, we have a commutative diagram in $\mathbf{D}^b(\modB)$,

\begin{displaymath}
 \xymatrix@C=3pc{
(\leftidx{_{B}}X_{A}\otimes^{\textbf{L}}_{A}X^{\vee}_{B})\otimes_{B}^{\textbf{L}}Y \ar[r]^-{\bar{\epsilon}_{Y}} & Y \ar[r] & \Cone(\bar{\epsilon}_{Y}) \ar[r] &  \\
 \leftidx{_{B}}(\mathscr{P}\otimes_{A} X^{\vee}_{B})\otimes_{B}^{\textbf{L}}Y \ar[u]^-{(\rho\otimes X^{\vee})\stackrel{\textbf{L}}{\otimes} Y}_-{\cong} \ar[r]^-{\bar{\epsilon}_{B}\otimes Y} &  B\otimes_{B}^{\textbf{L}} Y \ar[r] \ar[u]_-{\cong} & \Cone(\bar{\epsilon}_{B})\otimes_{B}^{\textbf{L}}Y \ar[r] \ar[u] & .
 }
\end{displaymath}
Hence $\Cone(\bar{\epsilon}_{Y})\in\perf(B)$ if and only if $\Cone(\bar{\epsilon}_{B})\otimes_{B}^{\textbf{L}}Y\in\perf(B)$. 
Similarly, if $Z$ is a complex in $\Der^{b}(\modA)$, one can show that $\Cone(\bar{\eta}_{Z})\in\perf(A)$ if and only if $\Cone(\bar{\eta}_{A})\otimes_{A}^{\textbf{L}}Z\in\perf(A)$.

The assumptions imply that $X^{\vee}\otimes_{B}^{\textbf{L}}X=\RHom_{B}(X,X)\cong A$ in $\mathbf{D}_{\mathrm{sg}}(A^e)$ and that $X\otimes_{A}^{\textbf{L}}X^{\vee}=\RHom_{A}(X^{\vee},X^{\vee})\cong B$ in $\mathbf{D}_{\mathrm{sg}}(B^e)$, thus $\Cone(\bar{\eta}_{A})\in\perf(A^{e})$ and $\Cone(\bar{\epsilon}_{B})\in\perf(B^{e})$, respectively. Therefore we see that the proof is finished if we employ Lemma~\ref{lemma:perf_env}.
\end{proof}

\begin{corollary}
\label{cor:homomorphisms}
Let $k$ be a field and let $A\rightarrow B$ be a homomorphism of finite dimensional $k$--algebras, where $B$ has finite projective dimension on both sides over $A$. If each of the following conditions is satisfied:
\begin{itemize}
\item[(i)] The cone of $A\rightarrow B$ in $\mathbf{D}^{b}(\mathrm{mod}(A^e))$ is perfect,
\item[(ii)] The cone of the natural map $B\otimes_{A}^{\textbf{L}}B\rightarrow B$ in $\mathbf{D}^{b}(\mathrm{mod}(B^e))$ is perfect,
\end{itemize}
then the functor $B\otimes_{A}^{\textbf{L}}-\colon \Dsing(A)\rightarrow \Dsing(B)$ is a triangulated equivalence. The converse holds in case $A/\rad(A)$ and $B/\rad(B)$ are separable.
\end{corollary}

\begin{proof}
This follows directly from Theorem~\ref{thm:suff_cond} if we consider the functor $B\otimes_{A}^{\textbf{L}}-\colon\Der^{b}(\modd(A))\rightarrow \Der^{b}(\modd(B))$, which is left adjoint to restriction of scalars.
\end{proof}

\begin{example}
\label{example:chen}
Let $k$ be a field, let $A$ be a finite dimensional $k$ algebra, and let $I$ be an ideal of $A$ which has finite projective dimension over $A^{e}$. Consider the canonical map $A\rightarrow A/I$. If condition (ii) of Corollary \ref{cor:homomorphisms} is satisfied (with $B:=A/I$), we obtain $\Dsing(A)\cong\Dsing(A/I)$. Condition (ii) is trivially satisfied if $I$ is an idempotent ideal and $\Tor_{\geqslant 1}^{A}(A/I,A/I)=0$, equivalently, if the canonical map $A/I\otimes_{A}^{\textbf{L}}A/I\rightarrow A/I$ is an isomorphism in $\Der^{b}(\mathrm{mod}(B^{e}))$. This is the content of a result of X.-W.~Chen~\cite{Chen}.
\end{example}

Recall that if $\Lambda$ is a finite dimensional algebra over a field $k$ and $e$ is an idempotent in $\Lambda$, there exists an adjoint pair
\begin{displaymath}
 \xymatrix@C=4pc{
\Der^{b}(\modd(e\Lambda e)) \ar@<0.9ex>[r]^-{\Lambda e \otimes^{\textbf{L}}_{e\Lambda e}-}  & \Der^{b}(\modd(\Lambda)) \ar[l]^-{e\Lambda \otimes_{\Lambda}-},
  }
\end{displaymath}
where the right adjoint $e\Lambda \otimes_{\Lambda}-$ is isomorphic to the functor $e(-)$, which is mupltiplication by $e$ and is an exact functor.

There is one more easy consequence of Theorem \ref{thm:suff_cond} which has been discussed (in a more general context) in \cite[Main Theorem~(ii)]{Psa3}.

\begin{proposition}
\label{prop:idemp_without_level}
Let $\Lambda$ be a finite dimensional $k$--algebra over a field $k$, let $e$ be an idempotent in $\Lambda$, and assume any of the following two:
\begin{itemize}
\item[(i)] $\pd_{(e\Lambda e)^{\op}}(\Lambda e)<\infty$ and $e\Lambda\in\proj(e\Lambda e)$, or
\item[(ii)]$\Lambda e\in\proj(e\Lambda e)^{\op}$ and $\pd_{e\Lambda e}(e\Lambda)<\infty$.
 \end{itemize}

Then in case $\pd_{\Lambda^{e}}(\Lambda/\Lambda e\Lambda)<\infty$ the functor $e\Lambda\otimes_{\Lambda}-$ induces a singular equivalence between $\Lambda$ and $e\Lambda e$. The converse holds if $\Lambda/\rad(\Lambda)$ and $e\Lambda e/\rad(e\Lambda e)$ are separable. 
\end{proposition}

\begin{proof}
Under any of the conditions (i) or (ii), Remark~\ref{rem:adjoints} implies that the above given adjunction restricts to one at the level of singularity categories. 
We observe that the unit of this adjunction is an isomorphism. Thus according to Theorem~\ref{thm:suff_cond}, if the following condition is satisfied,
\begin{equation}
\label{cone_iso}
\mbox{The cone of the natural map}\,\,\,\, \Lambda e \otimes_{e\Lambda e}^{\textbf{L}}e\Lambda\rightarrow \Lambda\,\,\,\, \mbox{in}\,\,\,  \Der^{b}(\modd(\Lambda^{e})) \,\,\, \mbox{is perfect}.
\end{equation}
we obtain a singular equivalence $\Lambda e \otimes^{\textbf{L}}_{e\Lambda e}-\colon\Dsing(e\Lambda e)\cong \Dsing(\Lambda)$, and the converse holds under the assumptions on separability. 
Under any of the conditions (i) or (ii) we obtain $\Lambda e \otimes_{e\Lambda e}^{\textbf{L}}e\Lambda\cong\Lambda e \otimes_{e\Lambda e}e\Lambda$, hence (\ref{cone_iso}) is satisfied if and only if the natural inclusion $\Lambda e \otimes_{e\Lambda e}e\Lambda=\Lambda e\Lambda\rightarrow \Lambda$ is an isomorphism in $\Dsing(\Lambda^{e})$, which is in turn equivalent to $\pd_{\Lambda^{e}}(\Lambda/\Lambda e\Lambda)<\infty$.
\end{proof}

We apply the above discussion in the context of Morita rings with zero bimodule maps; see \cite{Green} for a study of their homological properties.

\begin{example}
\label{example:Morita_matrix}
Let $A$ and $B$ be finite dimensional algebras over a field $k$ and consider two finitely generated bimodules $\leftidx{_B}M_{A}$ and $\leftidx{_A}N_{B}$. 
We consider the ring
\[\Lambda:=\begin{pmatrix} 
A & \leftidx{_A}N_{B} \\
\leftidx{_B}M_{A} & B
\end{pmatrix}
\]
with multiplication given by
\[ \left( \begin{array}{cc}
a & n \\
m & b
\end{array} \right) \cdot
\left( \begin{array}{cc}
a' & n' \\
m' & b'
\end{array} \right) :=
\left( \begin{array}{cc}
aa' & an'+nb' \\
ma'+bm' & bb'
\end{array} \right).
\]

Consider any of the following two conditions:
\begin{itemize}
\item[(i)] $\pd_{A^{\op}}M<\infty$ and $N\in\proj(A)$, or
\item[(ii)] $M\in\proj(A^{\op})$ and $\pd_{A}N<\infty$.
\end{itemize}
If we work with the idempotent $e=\begin{pmatrix} 
1 & 0 \\
0 & 0
\end{pmatrix}$, Proposition~\ref{prop:idemp_without_level} implies:
\[\pd_{\Lambda^{e}}(B)<\infty\,\,\, \Rightarrow \Lambda e \otimes^{\textbf{L}}_{A}-\colon\Dsing(A)\cong\Dsing(\Lambda),\]
and the converse implication holds in case $\Lambda/\rad(\Lambda)$ and $A/\rad(A)$ are separable.

Similarly, we may consider any of the following two conditions:
\begin{itemize}
\item[(i')] $\pd_{B^{\op}}N<\infty$ and $M\in\proj(B)$, or
\item[(ii')] $N\in\proj(B^{\op})$ and $\pd_{B}M<\infty$.
\end{itemize}
Then if we work with the idempotent $e=\begin{pmatrix} 
0 & 0 \\
0 & 1
\end{pmatrix}$, Proposition~\ref{prop:idemp_without_level} implies:
\[\pd_{\Lambda^{e}}(A)<\infty\,\,\, \Rightarrow \Lambda e \otimes^{\textbf{L}}_{B}-\colon\Dsing(B)\cong\Dsing(\Lambda),\]
and the converse implication holds in case $B/\rad(B)$ and $\Lambda/\rad(\Lambda)$ are separable. 
\end{example}

In Example \ref{example:morita_rings_level} below we show that the singular equivalences obtained in this example induce singular equivalences of Morita type with level. We study such equivalences in the next section.

\section{Singular equivalences of Morita type with level}
\label{sec:sing_eq_of_Morita_type}

The next definition, in the case $n=0$, was given by Brou\'{e}~\cite{Broue} in the study of equivalences of blocks of group algebras. The definition below is due to Wang \cite{Wang}.

\begin{definition}
\label{dfn-morita-type}
Let $k$ be a commutative ring and let $A$ and $B$ be two $k$-algebras which are projective as $k$--modules. Let $_{B}M_{A}$ and $_{A}N_{B}$ be bimodules such that, for some $l\in\mathbb{N}$, the following hold:
\begin{itemize}
\item[(i)] $M$ is finitely generated and projective as a $B$-module and as an $A^{\mathrm{o}}$-module,
\item[(ii)] $N$ is finitely generated and projective as an $A$-module and as a $B^{\mathrm{o}}$-module,
\item[(iii)] $N\otimes_{B}M\cong \Omega^{l}_{A^e}(A)$ in $\underline{\mathrm{mod}}(A^e)$,
\item[(iv)] $M\otimes_{A}N\cong \Omega^{l}_{B^e}(B)$ in $\underline{\mathrm{mod}}(B^e)$.
\end{itemize}
Then we say that the pair $(\leftidx{_{B}}M_{A},\leftidx{_{A}}N_{B})$ defines a \textit{singular equivalence of Morita type with level $l$} between $A$ and $B$.
\end{definition}

The following appears in \cite[Remark~2.2]{Wang}. We include a proof here.
\begin{proposition}
Let $k$ be a field and let $A$, $B$ be finite dimensional $k$--algebras. Assume that $(\leftidx{_{B}}M_{A},\leftidx{_{A}}N_{B})$ is a pair of bimodules that defines a singular equivalence of Morita type with level $l$. Then the functor $F:=M\otimes_{A}-\colon\mathbf{D}_{\mathrm{sg}}(A) \rightarrow  \mathbf{D}_{\mathrm{sg}}(B)$ is a triangulated equivalence with inverse $G:=\Sigma^{l}(-)\circ (N\otimes_{B}-)$. The case $l=0$ is stronger as it gives an equivalence $M\otimes_{A}-\colon\underline{\mathrm{mod}}(A)\rightarrow\underline{\mathrm{mod}}(B)$ with inverse $N\otimes_{B}-$.
\end{proposition}

\begin{proof}
The isomorphism (iii) in Definition \ref{dfn-morita-type} implies an isomorphism $(N\otimes_{B}M)\oplus P\cong \Omega^{l}_{A^e}(A)\oplus Q$ in $\mathrm{mod}(A^e)$, with $P$ and $Q$ in $\proj(A^{e})$. Let $X$ be a complex in $\Der^{b}(\modd(A))$. Then we have an isomorphism in $\Der^{b}(\modd(A))$ which is natural in $X$,
 \[(N\otimes_{B}M\otimes_{A}X)\oplus (P\otimes_{A}X) \cong(\Omega^{l}_{A^{e}}(A)\otimes_{A}X)\oplus(Q\otimes_{A}X).\]
We have that $P\otimes_{A}X$ and $Q\otimes_{A}X$ are in $\perf(A)$. Thus in $\Dsing(A)$, we obtain isomorphisms,
\[N\otimes_{B}M\otimes_{A}X\cong\Omega^{l}_{A^{e}}(A)\otimes_{A}X\cong\Omega^{l}_{A}(X)\cong\Sigma^{-l}(X),\] which are natural in $X$. This shows that $G\circ F\cong\mathrm{id}_{\Dsing(A)}$. Similarly one can show that $F\circ G\cong\mathrm{id}_{\Dsing(B)}$.
\end{proof}

\begin{remark}
In definition \ref{dfn-morita-type}, assume that $k$ is a field and that $A$ and $B$ have infinite global dimension. If the syzygies $\Omega^{l}_{A^{e}}(A)$ and $\Omega^{l}_{B^{e}}(B)$ are indecomposable over $A^{e}$ and $B^{e}$ respectively, then conditions $(iii)$ and $(iv)$ are respectively equivalent to:
\begin{itemize}
\item[(iii')] $N\otimes_{B}M\cong\Omega^{l}_{A^e}(A)\oplus X$, for some $X\in\mathrm{proj}(A^e)$,
\item[(iv')] $M\otimes_{A}N\cong \Omega^{l}_{B^e}(B)\oplus Y$, for some $Y\in\mathrm{proj}(B^e)$.
\end{itemize}

Indeed, we prove that (iii) implies (iii'): Since $N\otimes_{B}M\cong \Omega^{l}_{A^e}(A)$ in $\underline{\mathrm{mod}}(A^e)$, there exist projective $A^e$--modules $X'$ and $X''$ such that $(N\otimes_{B}M)\oplus X'\cong \Omega^{l}_{A^e}(A)\oplus X''$ in $\mathrm{mod}(A^e)$. Since $A$ has infinite global dimension we have that $\Omega^{l}_{A^e}(A)$ is a non-projective $A^e$--module \cite[1.5]{Ha}. Thus, from the Krull-Schmidt theorem we deduce that $N\otimes_{B}M\cong \Omega^{l}_{A^e}(A)\oplus X$, for some $X\in\mathrm{proj}(A^e)$. Similarly one can prove that (iv) implies (iv'). In the literature, in the case $l=0$, some authors define stable equivalences of Morita type using conditions (iii') and (iv').
\end{remark}

In the rest of the paper we will need the following.

\begin{fact}
\label{prop:description_cpx}
Under Setup~\ref{Setup}, let $X$ be a complex of finitely generated $B$-$A^{\op}$--bimodules which is perfect over $B$ and over $A^{\op}$, and  consider $\mathscr{P}\rightarrow X$ a projective resolution of $X$ over $B\otimes_{k}A^{\op}$. Write 
\[\mathscr{P}:=\cdots\rightarrow\mathscr{P}_{n+1}\xrightarrow{\partial_{n+1}^{\mathscr{P}}} \mathscr{P}_{n}\xrightarrow{\partial_{n}^{\mathscr{P}}} \mathscr{P}_{n-1}\cdots\rightarrow.\]
Then the complex $\mathscr{P}$ is isomorphic in $\Der^{b}(\modd(B\otimes_{k}A^{\op}))$ to a complex
\[\mathscr{L}= (0\rightarrow  \mathscr{L}_{s}\xrightarrow{\partial_{s}^{\mathscr{L}}} \mathscr{L}_{s-1}\rightarrow\cdots\rightarrow \mathscr{L}_{i}\rightarrow 0),\]
where for all $j=i,i+1,...,s-1;$\, $\mathscr{L}_{j}$ is finitely generated and projective as a $B$-$A^{\op}$--bimodule and $\mathscr{L}_{s}$ is finitely generated and projective over $B$ and $A^{\op}$ (but is not necessarily projective as a $B$-$A^{\op}$--bimodule). 

In case $X$ is a $B$-$A^{\op}$--bimodule concentrated in degree zero, with $\pd_{B}X=m$ and $\pd_{A^{\op}}X=n$, then for $s=\max\{n,m\}$ we may choose $\mathscr{L}_{s}\cong\Omega^{s}_{B\otimes_{k}A^{\op}}(X)$.
\end{fact}

\begin{proof}
See for instance \cite[Prop.~6.4.4]{Zim}.
\end{proof}

\begin{remark}
\label{rmk_idempot_with_level}
Let $\Lambda$ be a finite dimensional algebra over a field $k$. We will prove that the singular equivalences obtained in Proposition \ref{prop:idemp_without_level}, induce singular equivalences of Morita type with level. To see this, consider for instance the case where condition (i) in Proposition \ref{prop:idemp_without_level} holds and $\mathrm{pd}_{\Lambda^{e}}\Lambda/\Lambda e\Lambda<\infty$. We claim that for 
\begin{equation}
\label{level_in_prop}
l=\max\{\pd_{(e\Lambda e)^{\op}}(\Lambda e),\pd_{\Lambda^{e}}(\Lambda/\Lambda e\Lambda)\},
\end{equation}
the pair $(\Omega^{l}_{\Lambda\tens{k}(e\Lambda e)^{\op}}(\Lambda e),e\Lambda)$ defines a singular equivalence of Morita type with level $l$ between $\Lambda$ and $e\Lambda e$.

First note that  the bimodules $\leftidx{_{\Lambda}}\Omega^{l}_{\Lambda\tens{k}(e\Lambda e)^{\op}}(\Lambda e)_{e\Lambda e}$ and $\leftidx{_{e\Lambda e}} e\Lambda_{\Lambda}$ are finitely generated and projective on both sides (for the first we may apply Fact~\ref{prop:description_cpx}).

Moreover, if we consider the short exact sequence $0\rightarrow \Lambda e\Lambda \rightarrow \Lambda\rightarrow \Lambda/\Lambda e\Lambda\rightarrow 0$ of $\Lambda^{e}$--modules, after comparing projective resolutions, we deduce that for all $i\geqslant\pd_{\Lambda ^{e}}(\Lambda/\Lambda e\Lambda)$ we have that
$\Omega^{i}_{\Lambda^{e}}(\Lambda e\Lambda)\cong \Omega^{i}_{\Lambda^{e}}(\Lambda)$ in $\underline{\mathrm{mod}}(\Lambda^{e})$.

Thus for $l$ as in (\ref{level_in_prop}) there exists an isomorhism in $\underline{\mathrm{mod}}(\Lambda^{e})$,

\[\leftidx{_{\Lambda}}\Omega^{l}_{\Lambda\tens{k}(e\Lambda e)^{\op}}(\Lambda e)\tens{e\Lambda e} e\Lambda_{\Lambda}=\Omega_{\Lambda^{e}}^{l}(\Lambda e\Lambda)\cong\Omega_{\Lambda^{e}}^{l}(\Lambda),\]
and also an isomorphism in the stable category of $(e\Lambda e)$-$(e\Lambda e)^{\op}$--bimodules: 
\[e\Lambda\tens{\Lambda}\Omega^{l}_{\Lambda\tens{k}(e\Lambda e)^{\op}}(\Lambda e)\cong\Omega_{(e\Lambda e)^{e}}^{l}(e\Lambda e),\]
which finishes the claim.

If we assume that condition (ii) in Proposition~\ref{prop:idemp_without_level} holds and $\mathrm{pd}_{\Lambda^{e}}\Lambda/\Lambda e\Lambda<\infty$, then similarly one can prove that for \[l=\max\{\pd_{e\Lambda e}(e\Lambda),\pd_{\Lambda^{e}}(\Lambda/\Lambda e\Lambda)\},\]
the pair $(\Lambda e,\Omega^{l}_{e\Lambda e\tens{k}\Lambda^{\op}}(e\Lambda))$ defines a singular equivalence of Morita type with level $l$ between $\Lambda$ and $e\Lambda e$.
\end{remark}

In particular we obtain the following:

\begin{example}\textnormal{(cf. \cite[Sec.~3]{Wang})}
\label{example:morita_rings_level}
Let $A$ and $B$ be finite dimensional algebras over a field $k$ and let $\leftidx{_{B}}M_{A}$ and $\leftidx{_{A}}N_{B}$ be finitely generated bimodules. Consider the ring
\[\Lambda:=\begin{pmatrix} 
A & \leftidx{_A}N_{B} \\
\leftidx{_B}M_{A} & B
\end{pmatrix},
\]
as in Example \ref{example:Morita_matrix}. Then we have the following:
\begin{itemize}
\item[(i)] If $\pd_{\Lambda^{e}}(B)<\infty$, $\pd_{A^{\op}}M<\infty$ and $N\in\proj(A)$, then there is a singular equivalence of Morita type with level $l$ between $\Lambda$ and $A$, where $l=\max\{\pd_{A^{\op}}M,\pd_{B^{e}}B\}$. In fact, if $e=\begin{pmatrix} 
1 & 0 \\
0 & 0
\end{pmatrix}$ then the pair of bimodules which realizes this equivalence is $(\Omega^{l}_{\Lambda\tens{k}A^{\op}}(\Lambda e),e\Lambda)$.
\item[(ii)] If $\pd_{\Lambda^{e}}(A)<\infty$, $\pd_{B}M<\infty$ and $N\in\proj(B^{\op})$, then there is a singular equivalence of Morita type with level $l$ between $\Lambda$ and $B$, where $l=\max\{\pd_{B}M,\pd_{A^{e}}A\}$. In fact, if $e=\begin{pmatrix} 
0 & 0 \\
0 & 1
\end{pmatrix}$ then the pair of bimodules which realizes this equivalence is $(\Lambda e,\Omega^{l}_{B\tens{k}\Lambda^{\op}}(e\Lambda))$.
\end{itemize}
\end{example}

We continue with the proof of the main result which was stated in the introduction. We will need the following. 
\begin{lemma} 
\label{Lem:CM}
Let $k$ be a field and let $A$ and $B$ be finite dimensional $k$--algebras. Assume that $\leftidx{_{B}}M_{A}$ and $\leftidx{_{A}}N_{B}$ are bimodules which are finitely generated and projective on both sides. Then $M\otimes_{A}N$ is a maximal Cohen-Macaulay $B^{e}$--module and $N\otimes_{B}M$ is a maximal Cohen-Macaulay $A^{e}$--module.
\end{lemma}

\begin{proof}
We will only prove that $N\otimes_{B} M$ is a maximal Cohen--Macaulay $A^e$--module, that is, we claim that $\Ext_{A^e}^{\geqslant 1}(N\otimes_{B} M,A^e)=0$. We denote the $k$-dual $\Hom_{k}(-,k)$ by $D(-)$. We have the following $k$--linear isomorphisms:

\begin{eqnarray}
\RHom_{A^e}(N\otimes_{B} M,A\otimes_{k}A^o) & \cong & D\left((N\otimes_{B} M)\otimes_{A^e}^{\textbf{L}}D(A\otimes_{k}A^o)\right) \nonumber \\
 & \cong & D\left((N\otimes_{B} M)\otimes_{A^e}^{\textbf{L}}(D(A^o)\otimes_{k}D(A))\right)\nonumber\\
 & \cong & D\left(D(A)\otimes_{A}^{\textbf{L}} (N \otimes_{B}M)\otimes_{A}^{\textbf{L}} D(A^o)\right)\nonumber\\
 & \cong & D\left((D(A)\otimes_{A}^{\textbf{L}} N) \otimes_{B}(M\otimes_{A}^{\textbf{L}} D(A^o)\right)\nonumber\\
  & \cong & D\left((D(A)\otimes_{A} N) \otimes_{B}(M\otimes_{A} D(A^o)\right),\nonumber
\end{eqnarray}
where the last isomorphism holds since $N$ is projective as a left $A$--module and $M$ is projective as a right $A$--module. Hence the complex $\RHom_{A^e}(N\otimes_{B} M,A\otimes_{k}A^o)$ is homologically concentrated in degree zero, which proves the claim.
\end{proof}

\subsection{Proof of Theorem}
\label{proof_main} We consider the adjoint pair of functors $(F,G)$ as in Proposition~\ref{prop:adjoints}~(\ref{adjunction1}). Since $A$ and $B$ are Gorenstein, Lemma~\ref{lemma:Gor_adj} implies that the adjunction $(F,G)$ restricts to one at the level of singularity categories, $\bar{F}:=\Dsing(A)\rightleftarrows\Dsing(B):=\bar{G}$, which by assumption is a triangulated equivalence.

We employ a trick from \cite[Thm.~2.3]{Wang} (and its proof): Since the complex of $B$-$A^{\op}$--bimodules $X$ is perfect over $B$ and also perfect over $A^{\op}$, from Fact~\ref{prop:description_cpx}, a projective resolution of $X$ is isomorphic in $\mathbf{D}^b(\modd(B\otimes_{k}A^o))$ to a complex
\[\mathscr{L}:= (0 \rightarrow \mathscr{L}_{s}\rightarrow \mathscr{L}_{s-1}\rightarrow\cdots\rightarrow \mathscr{L}_{i}\rightarrow 0),\]
where, for $j=i,i+1,..., s-1;\, \mathscr{L}_{j}\in\proj(B\otimes_{k}A^{\op})$ and $\mathscr{L}_{s}\in\proj(B)\cap\proj(A^{\op})$.

Similarly, a projective resolution of the complex $X^{\vee}:=\RHom_B(X,B)$ is isomorphic, in the category $\mathbf{D}^b(\modd(A\otimes_{k}B^{\op}))$, to a complex 
\[\mathscr{W}:= (0\rightarrow \mathscr{W}_{s'}\rightarrow \mathscr{W}_{s'-1}\rightarrow\cdots\rightarrow \mathscr{W}_{i'}\rightarrow 0),\]
where, for $j=i',i'+1,..., s'-1;\, \mathscr{W}_{j}\in\proj(A\tens{k}B^{\op})$ and $\mathscr{W}_{s'}\in\proj(A)\cap\proj(B^{\op})$.

We consider the tensor product complex:
\[\mathscr{W}\otimes_{B} \mathscr{L} = (0\rightarrow Z_{s+s'}\rightarrow Z_{s+s'-1}\rightarrow \cdots Z_{u}\rightarrow 0),\] where for all $j\leq s+s'-1; \ Z_{j}$ is a finitely generated and projective over $A^{e}$. Note that $X^{\vee}\otimes_{B}^{\textbf{L}}X\cong\mathscr{W}\otimes_{B} \mathscr{L}$ in $\Der^{\mathrm{b}}(\modd(A^{e}))$.

Put $M:=\mathscr{L}_{s}$ and $N:=\mathscr{W}_{s'}$. In the singularity category $\mathbf{D}_{\mathrm{sg}}(A^{e})$, the ``hard truncation below'' at $s+s'$, which is the map
\begin{displaymath}
 \xymatrix@C=1pc{
 \mathscr{W}\otimes_{B}\mathscr{L}  \ar[d] ^-{\tau_{\geqslant s+s'}}  & 0 \ar[d] \ar[r] & Z_{s+s'} \ar@{=}[d] \ar[r] & Z_{s+s'-1}  \ar[r] \ar[d] & \cdots \ar[r] & Z_{u} \ar[d] \ar[r] & 0 \ar[d] \\
 \Sigma^{s+s'}(N\otimes_{B} M)  & 0 \ar[r] & N\otimes_{B} M \ar[r] & 0  \ar[r]  & \cdots \ar[r] & 0 \ar[r] & 0\, 
 }
\end{displaymath}
is an isomorphism (since $0\rightarrow Z_{s+s'-1}\rightarrow\cdots\rightarrow Z_{u}\rightarrow 0$ is perfect over $A^{e}$). Similarly, one can show that $\mathscr{L}\otimes_{A}\mathscr{W}\cong \Sigma^{s+s'}(M\otimes_{A}N)$ in $\Dsing(B^{e})$.

Since the adjunction $(\bar{F},\bar{G})$ is assumed to be a triangulated equivalence, Theorem~\ref{thm:suff_cond} gives an isomorphism $\bar{\eta}_{A}:A\rightarrow X^{\vee}\otimes_{B}^{\mathbf{L}}X $ in $\mathbf{D}_{\mathrm{sg}}(A^{e})$ and also an isomorphism $\bar{\epsilon}_{B}:B\leftarrow X\otimes_{A}^{\textbf{L}}X^{\vee} $ in $\mathbf{D}_{\mathrm{sg}}(B^{e})$ (for this we need the assumption on separability). 

Hence, in the singularity category $\mathbf{D}_{\mathrm{sg}}(A^{e})$, we obtain isomorphisms
\[\Sigma^{-(s+s')}A\xrightarrow{\cong}\Sigma^{-(s+s')}(\mathscr{W}\otimes_{B}\mathscr{L}) \xrightarrow{\cong}N\otimes_{B} M,\]
and also in the singularity category $\mathbf{D}_{\mathrm{sg}}(B^{e})$, we obtain isomorphisms
\[\Sigma^{-(s+s')}B\xrightarrow{\cong}\Sigma^{-(s+s')}(\mathscr{L}\otimes_{A}\mathscr{W}) \xrightarrow{\cong}M\otimes_{A} N.\]

Therefore we have $\Omega^{s+s'}_{A^e}(A)\cong N\otimes_{B} M$ in $\mathbf{D}_{\mathrm{sg}}(A^{e})$ and $\Omega^{s+s'}_{B^e}(B)\cong M\otimes_{A} N$ in $\mathbf{D}_{\mathrm{sg}}(B^{e})$.

Since $A$ is a Gorenstein algebra, the enveloping algebra $A^e$ is also Gorenstein \cite[Lemma~2.1]{zbMATH06267567}. Thus we can make use of the result of Buchweitz \cite{Buch} which gives a triangulated equivalence $F\colon\mathbf{D}_{\mathrm{sg}}(A^{e})\xrightarrow{\cong} \underline{\mathrm{MCM}}(A^e)$, where $\underline{\mathrm{MCM}}(A^e)$ denotes the stable category of maximal Cohen-Macaulay $A^e$--modules. The functor $F$ maps any $A^e$--module, viewed as a complex concentrated in degree zero, to its maximal Cohen--Macaulay approximation, see \cite[Thm.~5.1.2]{Buch}.

Note that Lemma \ref{Lem:CM} informs us that $N\otimes_{B} M$ is in $\mathrm{MCM}(A^e)$ and that $M\otimes_{A} N$ is in $\mathrm{MCM}(B^e)$. We mention that, if necessary, we may pick large enough indices $s, s'$ in order to obtain an isomorphism 
$\Omega^{s'+s}(A)\cong N\otimes_{B} M$ in the category $\underline{\mathrm{MCM}}(A^e)$ and an isomorphism $\Omega^{s'+s}(B)\cong M\otimes_{A} N$ in the category $\underline{\mathrm{MCM}}(B^e)$.

\subsection{Proof of Corollary}\label{proof_corollary} We first prove that there is a well-defined functor $M\otimes_{A}-\colon \mathrm{MCM}(A)\rightarrow \mathrm{MCM}(B)$. Indeed, if $N$ is in $\mathrm{MCM}(A)$ then we need to prove that the complex $\RHom_{B}(M\otimes_{A}N,B)$, or equivalently $\RHom_{A}(N,\Hom_{B}(M,B))$, is homologically concentrated in degree zero. 
Let $B\xrightarrow{\sim} I$ be an (augmented) injective resolution (of finite length) of $B$ over itself. 
Since $M$ is projective on both sides, the functor $\Hom_{B}(M,-)$ is exact and maps injectives over $B$ to injectives over $A$. 
Therefore $\Hom_{B}(M,B)$ is an $A$--module of finite injective dimension, and since $A$ is Gorenstein, from \cite[2.3.2/2.3.5]{Chen-Survey} for instance we obtain that $\Ext^{\geqslant 1}_{A}(N,\Hom_{B}(M,B))=0$. 

Moreover, since $M$ is projective over $B$ it follows easily that there is an induced functor $M\otimes_{A}-\colon \underline{\mathrm{MCM}}(A)\rightarrow \underline{\mathrm{MCM}}(B)$. 

We consider the functor $M\otimes_{A}-\colon \Der^{b}(\modA)\rightarrow \Der^{b}(\modB)$. From Buchweitz~\cite{Buch}, the given equivalence $M\otimes_{A}-\colon\underline{\mathrm{MCM}}(A)\cong\underline{\mathrm{MCM}}(B)$ induces an equivalence $M\otimes_{A}-\colon\Dsing(A)\cong\Dsing(B)$, so we fall under the assumptions of the main Theorem.

Put $l:=2\max\{\mathrm{vdim}(A),\mathrm{vdim}(B)\}$ and let $N:=\Omega^{l}_{A\otimes_{k}B^{\op}}(M^{\vee})$. 
We claim that there exists an isomorphism
 \[N\otimes_{B}M\cong \Omega^{l}_{A^{e}}(A)\,\,\,\,\,\,  \mbox{in}\,\,\,\,\,\,  \underline{\mathrm{MCM}}(A^{e}),\]
 as well as an isomorphism,
\[M\otimes_{A} N\cong\Omega^{l}_{B^{e}}(B)\,\,\,\,\,\,  \mbox{in}\,\,\,\,\,\,  \underline{\mathrm{MCM}}(B^{e}).\]
First note that, by \cite[Lemma~2.1]{zbMATH06267567}, Gorensteinness implies that $\mathrm{vdim}(A^e)\leqslant 2\mathrm{vdim}(A)$ and $\mathrm{vdim}(B^e)\leqslant 2\mathrm{vdim}(B)$; thus this choice of $l$ guarantees that $\Omega^{l}_{A^{e}}(A)$ is in $\mathrm{MCM}(A^e)$ and that $\Omega^{l}_{B^{e}}(B)$ is in $\mathrm{MCM}(B^e)$.

In addition, we note that $M^{\vee}$ is projective over $B^{\mathrm{o}}$ (but of course it might not be projective over $A$). Thus, keeping the same notation as in the proof of \ref{proof_main}, in order to construct the desired complex 
\[\mathscr{W}:= (0\rightarrow \mathscr{W}_{s'}\rightarrow \mathscr{W}_{s'-1}\rightarrow\cdots\rightarrow \mathscr{W}_{i'}\rightarrow 0),\]
which is a soft truncation of a projective resolution of $M^{\vee}$ over $A\otimes_{k}B^{\mathrm{o}}$, we may pick any $s'\geq \mathrm{pd}_{A}M^{\vee}$. Since Gorensteiness implies that  $\mathrm{pd}_{A}M^{\vee}\leqslant\mathrm{vdim}(A)$ \cite[Lemma 2.3.2]{Chen-Survey}, we see that for $l$ chosen as above the proof of \ref{proof_main} carries on and produces the desired isomorphisms.

\section*{Acknowledgements}
I am grateful to an anonymous referee for providing various comments that improved the quality of the paper. I should also like to thank Chrysostomos Psaroudakis for providing comments on a preliminary version.


\end{document}